\documentclass[10pt]{amsart}
\usepackage{amsmath,amssymb,amsthm,bbm}
\renewcommand{\epsilon}{\varepsilon}
\renewcommand{\phi}{\varphi}
\renewcommand{\kappa}{\varkappa}
\renewcommand{\tilde}{\widetilde}
\newcommand{\IE}{\mathbb{E}}
\newcommand{\I}{\mathbbm{1}}
\newcommand{\IN}{\mathbb{N}}
\newcommand{\IP}{\mathbb{P}}
\newcommand{\IR}{\mathbb{R}}
\newcommand{\IZ}{\mathbb{Z}}
\newcommand{\cal}{\mathcal}
\newtheorem{theorem}{Theorem}
\newtheorem{lemma}[theorem]{Lemma}

 \newcommand{\brJ}{\bar J}
 
 \newcommand{\brl}{\bar\ell}

 \newcommand{\cF}{\mathcal{F}}

 \newcommand{\htheta}{\hat\theta}
 
\newcommand{\tV}{\tilde V}

\begin{document}
\title[Excited random walks]{Scaling limits of recurrent excited
  random walks on integers}

\author{Dmitry Dolgopyat and Elena Kosygina}

\begin{abstract}
  We describe scaling limits of recurrent excited random walks (ERWs)
  on $\IZ$ in i.i.d.\ cookie environments with a bounded number
  of cookies per site. We allow both positive and negative
  excitations. It is known that ERW is recurrent if and only if the
  expected total drift per site, $\delta$, belongs to the interval
  $[-1,1]$. We show that if $|\delta|<1$ then the diffusively scaled
  ERW under the averaged measure converges to a
  $(\delta,-\delta)$-perturbed Brownian motion. In the boundary case,
  $|\delta|=1$, the space scaling has to be adjusted by an extra
  logarithmic term, and the weak limit of ERW happens to be a constant
  multiple of the running maximum of the standard Brownian motion, a
  transient process.
\end{abstract}
\maketitle
\section{Introduction and main results}
Given an arbitrary positive integer $M$ let 
\begin{align*}
  \Omega_M:=\big\{((\omega_z(i))_{i\in\IN})_{z\in\IZ}
 \mid \, \omega_z(i)\in[0,1],\ &\text{for }
   i\in\{1,2,\dots,M\}\\
  \text{\ and}\  
  \omega_z(i)=1/2,\  &\text{for } i>M,\  z\in\IZ\big\}.
\end{align*}
An element of $\Omega_M$ is called a cookie environment. For each
$z\in\IZ$, the sequence $\{\omega_z(i)\}_{i\in\IN}$ can be thought of
as a stack of cookies at site $z$. The number $\omega_z(i)$  represents the 
transition probability from $z$ to $z+1$ of a nearest-neighbor random
walk upon the $i$-th visit to $z$. If $\omega_z(i)\ge 1/2$ (resp.
$\omega_z(i)<1/2$) the corresponding cookie is called non-negative
(resp.\ negative).

Let $\IP$ be a probability measure on $\Omega_M$, which satisfies the
following two conditions: 
\begin{itemize}
\item[(A1)]  Independence: the sequence 
$(\omega_z(\cdot))_{z\in\IZ}$ is i.i.d.\ under $\IP$;
\item [(A2)] Non-degeneracy: 
$\IE\left[\prod_{i=1}^M
  \omega_0(i)\right]>0\ \text{ and }\ \IE\left[\prod_{i=1}^M
  (1-\omega_0(i))\right]>0$.
\end{itemize}

For $x\in\IZ$ and $\omega\in \Omega_M$ consider an integer valued
process $X:=(X_j),\ j\ge 0$, on some probability
space $({\cal X},\mathcal{F},P_{x,\omega})$, which $P_{x,\omega}$-a.s.\
satisfies $P_{x,\omega}(X_0=x)=1$ and
\begin{equation*}
  P_{x,\omega}(X_{n+1}=X_n+1\,|\,{\cal
    F}_n)=1-P_{x,\omega}(X_{n+1}=X_n-1\,|\,{\cal F}_n)=\omega_{X_n}(L_{X_n}(n)),
\end{equation*}
where ${\cal F}_n\subset {\cal F}$, $n\ge 0$, is the natural
filtration of $X$ and $L_m(n):=\sum_{j=0}^n\I_{\{X_j=m\}}$ is the
number of visits to site $m$ by $X$ up to time $n$. Informally
speaking, upon each visit to a site the walker eats the topmost cookie
from the stack at that site and makes one step to the right or to the
left with probabilities prescribed by this cookie. The consumption of
a cookie $\omega_z(i)$ induces a drift of size
$2\omega_z(i)-1$
. Since $\omega_z(i)=1/2$ for all $i>M$, the walker will make unbiased
steps from $z$ starting from the $(M+1)$-th visit to $z$. Let $\delta$
be \textit{the expected total drift per site}, i.e.
\begin{equation}
  \label{D}
  \delta\ :=\ \IE\Bigg[\sum_{i\ge 1}(2\omega_0(i)-1)\Bigg]\ 
=\ \IE\left[\sum_{i=1}^M(2\omega_0(i)-1)\right].
\end{equation}
The parameter $\delta$ plays a key role in the classification of the
asymptotic behavior of the walk. For a fixed $\omega\in\Omega$ the
measure $P_{\omega,x}$ is called \textit{quenched}. The
\textit{averaged} measure $P_x$ is obtained by averaging over environments,
i.e.\ $P_x(\ \cdot\ ):=\IE\left(P_{x,\omega}(\ \cdot\ )\right)$.

There is an obvious symmetry between positive and negative cookies: if
the environment $(\omega_z)_{z \in \IZ}$ is replaced by
$(\omega^\prime_z)_{z \in \IZ} $ where $\omega^\prime_z (i)=1-
\omega_z (i)$, for all $i\in\IN,\ z\in\IZ$, then $X'$, the ERW
corresponding to the new environment, satisfies $
X'\overset{\mathrm{d}}{=} -X$, where $\overset{\mathrm{d}}{=}$
denotes the equality in distribution.  Thus, it is sufficient to
consider only non-negative $\delta$ (this, of course, allows
both negative and positive cookies), and we shall always assume this
to be the case.

ERW on $\IZ$ in a non-negative cookie environment and its natural
extension to $\IZ^d$ (when there is a direction in $\IR^d$ such that
the projection of a drift induced by every cookie on that direction is
non-negative) were considered previously by many authors (see, for
example, \cite{BW03}, \cite{Ze05}, \cite{Ze06}, \cite{BS08a},
\cite{BS08b}, \cite{MPV06} \cite{BR07}, \cite{Do11}, \cite{MPRV}, and
references therein).

Our model allows both positive and negative cookies but restricts
their number per site to $M$. This model was studied in \cite{KZ},
\cite{KM}, \cite{RR}, \cite{Pe}. It is known that the process is
recurrent (i.e.\ for $\IP$-a.e.\ $\omega$ it returns to the starting
point infinitely often) if and only if $\delta\le 1$ (\cite{KZ}). For
transient (i.e.\ not recurrent) ERW, there is a rich variety of limit
laws under $P_0$ (\cite{KM}).

In this paper we study scaling limits of recurrent ERW under $P_0$.
The functional limit theorem for recurrent ERW in stationary ergodic
non-negative cookie environments on strips $\IZ\times(\IZ/L\IZ)$,
$L\in\IN$, under the quenched measure was proven in \cite{Do11}. Our
results deal only with i.i.d.\ environments on $\IZ$ with bounded
number of cookies per site but remove the non-negativity assumption on
the cookies. We are also able to treat the boundary case $\delta=1$.
Extensions of these results and results of \cite{KM} to strips, or
$\IZ^d$ for $d>1$, or the ``boundary'' case for the model treated
in \cite{Do11} are still open problems.

To state our results we need to define the candidates for limiting
processes. Let $D([0,\infty))$ be the Skorokhod space of c\`adl\`ag
functions on $[0,\infty)$ and denote by $\overset{J_1}{\Rightarrow}$
the weak convergence in the standard ($J_1$) Skorokhod topology on
$D([0,\infty))$. Unless stated otherwise, all
processes start at the origin at time $0$. Let $B=(B(t)),\ t\ge 0$,
denote a standard Brownian motion and
$X_{\alpha,\beta}=(X_{\alpha,\beta}(t)),\ t\ge 0$, be an $(\alpha,
\beta)$-perturbed Brownian motion, i.e.\ the solution of the equation
\begin{equation}
  \label{abp}
  X_{\alpha,\beta}(t)=B(t)+\alpha\sup_{s\leq t}
X_{\alpha,\beta}(s)+\beta \inf_{s\leq t} X_{\alpha,\beta}(s), 
\end{equation}
For $(\alpha,\beta)\in(-\infty,1)\times(-\infty,1)$ the equation
(\ref{abp}) has a pathwise unique solution that is adapted to the
filtration of $B$ and is a.s.\ continuous (\cite{PW}, \cite{CD}).  Now
we can state the results of our paper.
\begin{theorem}[Non-boundary case]
\label{ThERWRecLim}
If $\delta\in[0,1)$ then \[\dfrac{X_{[n\cdot]}}{\sqrt
  n}\overset{J_1}{\Rightarrow} X_{\delta,-\delta}(\cdot)\ \text{as }n\to\infty. \]
\end{theorem}
We note that there are other known random walk models which after
rescaling converge to a perturbed Brownian motion (see, e.g., \cite{Da,
  T}).
\begin{theorem}[Boundary case]
\label{ThERWCrit}
Let $\delta=1$ and $B^*(t):=\max_{s\le t}B(s)$, $t\ge 0$. Then there exists a
constant $D>0$ such that
\[\frac{X_{[n\cdot]}}{D \sqrt{n}\log n}\overset{J_1}{\Rightarrow} B^*(\cdot)\
\text{as }n\to\infty.\]
\end{theorem}
Observe that for $\delta=1$ the limiting process is transient while
the original process is recurrent. To prove Theorem~\ref{ThERWCrit} we
consider the process $S_j:=\max_{\, 0\le i\le j}X_i$,
$j\ge 0$, and show that after rescaling it converges to the
running maximum of Brownian motion.  The stated result then comes from the
fact that with an overwhelming probability the maximum amount of
``backtracking'' of $X_j$ from $S_j$ for $j\le [Tn]$ is of order
$\sqrt{n}$, which is negligible on the scale $\sqrt{n}\log n$ (see
Lemma~\ref{btrack}).

\section{Notation  and preliminaries}

Assume that $\delta\ge 0$ and $X_0=0$. Let $T_x=\inf\{j\ge 0:\,X_j=x\}$ be the
first hitting time of $x\in \mathbb{Z}$. Set
\[S_n=\max_{k\leq n}
X_k,\quad I_n=\min_{k\leq n}
X_k,\quad R_n=S_n-I_n+1, \quad n\ge 0.\]
 At first, we recall the connection with branching processes exploited
in \cite{BS08a}, \cite{BS08b}, \cite{KZ}, and \cite{KM}.

For $n\in\IN$ and $0\le k\le n$ define 
\[D_{n,k}=\sum_{j=0}^{T_n-1}\I_{\{X_j=k,\ X_{j+1}=k-1\}},\]
the number of jumps from $k$ to $k-1$ before time $T_n$. Then
\begin{equation}\label{Tn}
T_n=n+2\sum_{k\le n}D_{n,k}=n+2\sum_{0\le k\le n}D_{n,k}+2\sum_{k<0}D_{n,k}.
\end{equation} 
Consider the ``backward'' process
$\left(D_{n,n},D_{n,n-1}\dots,D_{n,0}\right)$. Obviously, $D_{n,n}=0$
for every $n\in\IN$. Moreover, given
$D_{n,n},D_{n,n-1},\dots,D_{n,k+1}$, we can write 
\begin{align*}
 D_{n,k}=&\sum_{j=1}^{D_{n,k+1}+1}(\#\ \text{of jumps from $k$ to 
$k-1$ between the $(j-1)$-th}\\&\text{
   and $j$-th jump from $k$ to $k+1$ before time $T_n$}),\ k=0,1,\dots,n-1.
\end{align*}
Here we used the observation that the number of jumps from $k$ to
$k+1$ before time $T_n$ is equal to $D_{n,k+1}+1$ for all $0\le k \le
n-1$. It follows from the definition that
$\left(D_{n,n},D_{n,n-1}\dots,D_{n,0}\right)$ is a Markov
process. Moreover, it can be recast as a branching process with
migration (see \cite{KZ}, Section 3, as well as \cite{KM}, Section
2). Let $V:=(V_k),\ k\ge 0$, be the process such that $V_0=0$
and \[(V_0,V_1,\dots,V_n)\overset{\mathrm{d}}{=}
\left(D_{n,n},D_{n,n-1}\dots,D_{n,0}\right)\quad \text{ for all
  $n\in\IN$}.\] Denote by $\sigma\in[1,\infty]$ and
$\Sigma\in[0,\infty]$ respectively the lifetime and the total progeny
over the lifetime of $V$, i.e.\  $
  \sigma=\inf\{k>0\,:\, V_k=0\},
\  \Sigma=\sum_{k=0}^{\sigma-1}V_k$.
The probability measure that corresponds to $V$ will be denoted by
$P_0^V$. The following result will be used several times throughout
the paper.
\begin{theorem}[\cite{KM}, Theorems 2.1 and 2.2]
\label{LmERWLongExc} 
Let $\delta>0$. Then
  \begin{align}
    &\lim_{n\to\infty}n^\delta P_0^V(\sigma>n)=C_1\in(0,\infty);\label{lifetime}\\
    &\lim_{n\to\infty}n^\delta
    P_0^V\left(\Sigma>n^2\right)=C_2\in(0,\infty).\label{progeny}
  \end{align}
\end{theorem}
We shall need to consider $V$ over many lifetimes. Let $\sigma_0=0$,
$\Sigma_0=0$,
\begin{equation}
  \label{sigmas}
  \sigma_i=\inf\{k>\sigma_{i-1}:\,V_k=0\},\quad
  \Sigma_i=\sum_{k=\sigma_{i-1}}^{\sigma_i-1}V_k,\quad i\in\IN.
\end{equation}
Then $(\sigma_i-\sigma_{i-1},\Sigma_i)_{i\in\IN}$ are i.i.d.\ under
$P^V_0$,
$(\sigma_i-\sigma_{i-1},\Sigma_i)\overset{\mathrm{d}}{=}(\sigma,\Sigma)$,
$i\in\IN$.

\section{Non-boundary case: two useful lemmas}
Let $\delta\in[0,1)$. First of all, we show that by time $n$ the
walker consumes almost all the drift between $I_n$ and $S_n$.
\begin{lemma}
\label{LmERWEatAll}
Assume that $\delta\in[0,1)$. Given $\gamma_1>\delta$, there exist
$\gamma_2>0$ and $\theta\in(0,1)$ such that for
all $1\le \ell\le n$ 
\begin{align}
\label{EatRight}
&P_0\left(\sum_{m=n-\ell }^{n-1}\I_{\{L_m(T_n)<M\}}> \ell ^{\gamma_1}\right)\le
\theta^{\ell ^{\gamma_2}}\quad\text{and}\\
\label{EatLeft}
&P_0\left(\sum_{m=-(n-1)}^{-(n-\ell )}\I_{\{L_m(T_{-n})<M\}}> 
\ell ^{\gamma_1}\right)\le
\theta^{\ell ^{\gamma_2}}.
\end{align}
\end{lemma}
\begin{proof}
  We shall start with \eqref{EatRight} and use the connection with
  branching processes. Since the event we are interested in depends
  only on the environment and the behavior of the walk on $\{n-\ell
  ,n-\ell +1,\dots\}$, we may assume without loss of generality that
  the process starts at $n-\ell $ and, thus, by translation invariance
  consider only the case $\ell =n$.

  Let $L^V_k(n)=\sum_{j=0}^n \I_{\{V_j=k\}}$. We have
  \begin{multline}
    P_0\left(\sum_{m=0}^{n-1}\I_{\{L_m(T_n)<M\}}> n
      ^{\gamma_1}\right)\le P_0\left(\sum_{m=0
      }^n\I_{\{D_{n,m}<M\}}> n^{\gamma_1}\right)= P^V_0\left(\sum_{m=0
      }^n\I_{\{V_m<M\}}> n^{\gamma_1}\right)\\\le M\max_{0\le
      k<M}P^V_0\left(\sum_{m=0}^n\I_{\{V_m=k\}}> \frac{n
        ^{\gamma_1}}{M}\right)\label{v}= M\max_{0\le k<M}P^V_0\left(L^V_k(n)>
      \frac{n ^{\gamma_1}}{M}\right).
  \end{multline}
  At first, consider the case $\delta\in (0,1)$. Let $k=0$. Then (see
  (\ref{lifetime}) and (\ref{sigmas})) for all sufficiently large $n$
  we get
  \[P^V_0\left(L^V_0(n)> \frac{n ^{\gamma_1}}{M}\right)\le
  \prod_{i=1}^{[n^{\gamma_1}/M]} P^V_0\left(\sigma_i-\sigma_{i-1}\le
    n\right)\le\left(1-\frac{C_1}{2n^\delta}\right)^{[n^{\gamma_1}/M]}.\]
  Since $\gamma_1>\delta$, this implies the desired estimate for
  $k=0$.

 Let $k\in\{1,2,\dots,M-1\}$. Then for any $\epsilon>0$
  \begin{multline*}
    P^V_0\left(L^V_k(n)> \frac{n
      ^{\gamma_1}}{M}\right)=\\P^V_0\left(L^V_k(n)> \frac{n
      ^{\gamma_1}}{M}, L^V_0(n)> \frac{\epsilon n
      ^{\gamma_1}}{2M}\right)+P^V_0\left(L^V_k(n)> \frac{n
      ^{\gamma_1}}{M}, L^V_0(n)\le \frac{\epsilon n
      ^{\gamma_1}}{2M}\right)
  \\\le P^V_0\left(L^V_0(n)> \frac{\epsilon n
      ^{\gamma_1}}{2M}\right)+P^V_0\left(L^V_0(n)\le \frac{\epsilon n
      ^{\gamma_1}}{2M}\,\Big|\,L^V_k(n)> \frac{n
      ^{\gamma_1}}{M}\right).
  \end{multline*}
  We only need to estimate the last term. Notice that by (A2) there is
  $\epsilon >0$ such that $P^V_0(V_{j+1}=0\,|\,V_j=k)\ge \epsilon$ for
  all $k\in\{1,2,\dots,M-1\}$ and $j\in\mathbb{N}$. Therefore, the
  last term is bounded above by the probability that in at least
  $[n^{\gamma_1}/M]$ independent Bernoulli trials with probability of
  success in each trial of at least $\epsilon$ there are at most
  $[\epsilon n^{\gamma_1}/(2M)]$ successes. This probability is
  bounded above by $\exp(-cn^{\gamma_1}/M)$ for some positive
  $c=c(\epsilon)$. This completes the proof of (\ref{EatRight}) for
  $\delta>0$.

  If $\delta=0$ we modify the environment by increasing slightly the
  drift (to the right) in the first cookie at each site.  Let $\tV$ be the
  branching process corresponding to the modified
  environment. There is a natural coupling between $V$ and $\tV$ such
  that $\tV_j\leq V_j$, $j\in\{0,1,\dots,n\}$.  Accordingly, 
\[\sum_{j=0}^n\I_{\{V_j< M\}}\leq \sum_{j=0}^n\I_{\{\tV_j< M\}},\] 
and (\ref{EatRight}) for $\delta=0$ follows from the
result for $\delta>0$ and the second line of (\ref{v}). 

Next after replacing $X$ by $-X$ proving \eqref{EatLeft} reduces to
proving \eqref{EatRight} for $\delta\leq 0$ and $\gamma_1>0$. As
above, the result for $\delta\leq 0$ can be deduced from the result
for $\delta\in(0,\gamma_1)$ by coupling of the corresponding branching
processes.
\end{proof}

Next we show that $\sqrt{n}$ is a correct scaling in Theorem \ref{ThERWRecLim}.
\begin{lemma}
\label{LmERWRange}
Assume that $\delta\in[0,1)$. There exists $\theta\in(0,1)$ such that
for all $L>0$, $\ell\in\IN\cup\{0\}$, and $n\in\IN$
\[P_0\left(T_{\ell+n}-T_\ell\leq \frac{n^2}{L}\right)\leq \theta^{\sqrt{L}} \quad
\text{and}\quad P_0\left(T_{-\ell-n}-T_{-\ell}\leq \frac{n^2}{L}\right)\leq \theta^{\sqrt{L}}. \]
\end{lemma}

\begin{proof}
  We shall prove the first inequality for $\delta\in(0,1)$. The case
  $\delta=0$ and the second inequality are handled in exactly the same
  way as in the proof of Lemma~\ref{LmERWEatAll}.

  Since $T_{n+\ell}-T_\ell\ge \sum_{k=\ell}^{n+\ell}
  D_{n+\ell,k}\overset{\mathrm{d}}{=}\sum_{j=0}^n V_j$, it is
  enough to show that
  \[ P^V_0\left(\sum_{j=0}^n V_j\leq \frac{n^2}{L}\right)\leq
  \theta^{\sqrt{L}}.\] 
Notice that by the Markov property and the stochastic monotonicity of $V$ in
the initial number of particles
\begin{multline}
  P^V_0\left(\sum_{j=0}^{m+k}V_j\le n\right)\le
  P^V_0\left(\sum_{j=m+1}^{m+k}V_j\le n\,\Big|\,\sum_{j=0}^m V_j\le
    n\right) P^V_0\left(\sum_{j=0}^m V_j\le
    n\right)\\\le P^V_0\left(\sum_{j=0}^k V_j\le
    n\right)P^V_0\left(\sum_{j=0}^m V_j\le
    n\right).\label{split}
\end{multline}
Suppose that we can show that there exist $K,n_0\in\IN$
  such that for all $n\ge n_0$
\begin{equation}
\label{CorrectScale}
P^V_0\left(\sum_{j=0}^{Kn} V_j\leq n^2\right)\leq \frac{1}{2}.
\end{equation}
Then using (\ref{split}) and (\ref{CorrectScale}) we get that for all
$L>4K^2$ and $n\ge \sqrt{L}n_0$
\begin{multline*}
  P_0^V\left(\sum_{j=0}^n V_j\le \frac{n^2}{L}\right)\le
  \left(P^V_0\left(\sum_{j=0}^{[2Kn/\sqrt{L}]}V_j\le
      \frac{n^2}{L}\right)\right)^{[\sqrt{L}/(2K)]}\\ \le
  \left(P^V_0\left(\sum_{j=0}^{2K[n/\sqrt{L}]}V_j\le
      4\left[\frac{n}{\sqrt{L}}\right]^2\right)\right)^{[\sqrt{L}/(2K)]}
  \le\left(\left(\frac12\right) ^{1/(4K)}\right)^{\sqrt{L}},
\end{multline*}
and we are done.

To prove (\ref{CorrectScale}), we observe that due to (\ref{lifetime})
the sequence $\sigma_m/m^{1/\delta}$, $m\in\IN$, has a limiting
distribution (\cite{Du}, Theorem 3.7.2) and, thus, if $K$ is large
then $ P_0(\sigma_{[(\sqrt{K} n)^\delta]}> Kn)\leq 1/4$ for all large
enough $n$.  We conclude that there is an $n_0\in\IN$ such that for
all $n\ge n_0$
\begin{multline*}
  P^V_0\left(\sum_{j=0}^{Kn} V_j\leq n^2\right)\leq
  \frac14+P^V_0\left(\sum_{j=0}^{\sigma_{[(\sqrt{K} n)^\delta]}} V_j\leq n^2,
    \sigma_{[(\sqrt{K} n)^\delta]}\le
    Kn\right)\\\leq \frac14+P^V_0\left(\sum_{i=1}^{[(\sqrt{K}
      n)^\delta]}\Sigma_i\le n^2\right) \le
  \frac14+\prod_{i=1}^{[(\sqrt{K}
      n)^\delta]}P^V_0\left(\Sigma_i\le n^2\right)
\overset{\text{(\ref{progeny})}}{\le}
    \frac14+ \left(1-\frac{C_2}{2n^\delta}\right)^{[(\sqrt{K}
      n)^\delta]}.
\end{multline*}
This immediately gives (\ref{CorrectScale}) if $K$ is chosen
sufficiently large.
\end{proof}

\section{Non-boundary case: Proof of Theorem~\ref{ThERWRecLim}}
Let $\Delta_n=X_{n+1}-X_n$ and 
\begin{equation}
  \label{bcd}
  B_n=\sum_{k=0}^{n-1} 
  \left(\Delta_k-E_{0,\omega}(\Delta_k|\cF_k)\right), \quad
  C_n=\sum_{k=0}^{n-1} E_{0,\omega}(\Delta_k|\cF_k).
\end{equation}
Then $X_n=B_n+C_n$, where $(B_n),\ n\ge 0$ is a martingale. 
Define
\[X^{(n)}(t):=\frac{X_{[nt]}}{\sqrt n},\quad
B^{(n)}(t):=\frac{B_{[nt]}}{\sqrt n},\quad
C^{(n)}(t):=\frac{C_{[nt]}}{\sqrt n},\quad t\ge 0,\ n\in\IN.\]
Theorem~\ref{ThERWRecLim} is an easy consequence of the following
three lemmas, the first of which holds for the quenched and the last
two for the averaged measures.
\begin{lemma}
  \label{Bn} Let $B$ be a standard Brownian motion. Then
  $B^{(n)}\overset{J_1}{\Rightarrow} B$ as $n\to\infty$ for
  $\IP$-a.e.\ $\omega$.

\end{lemma}
\begin{lemma}\label{Cn} For each $t\ge 0$ and
  $\epsilon>0$
  \[P_0\left(\sup_{k\leq nt} \frac{|C_k-\delta R_k|}{\sqrt
      n}>\epsilon\right)\to 0. \]
\end{lemma}
\begin{lemma}
\label{LmERWTight}
The sequence $X^{(n)}$, $n\ge 1$, is tight in
$D([0,\infty))$. Moreover, if $X$ is a limit point of
this sequence and $P$ is the corresponding measure on $D([0,\infty))$
then $P(X\in C([0,\infty)))=1$.
\end{lemma}
\begin{proof}[Proof of Theorem~\ref{ThERWRecLim} assuming
  Lemmas~\ref{Bn}--\ref{LmERWTight}]
  Since $X^{(n)}$, $n\ge 1$, is tight and $B^{(n)}\overset{J_1}{\Rightarrow} B$ as
  $n\to\infty$, the sequence $C^{(n)}$, $n\ge 1$, as the difference of two
  tight sequences is also tight.  We can assume
  by choosing a subsequence that $X^{(n)}\overset{J_1}{\Rightarrow} X$, where $X$ is
  continuous by Lemma~\ref{LmERWTight}. The mapping $x(\cdot)\mapsto
  r^x(\cdot):=\sup_{s\le \cdot}x(s)-\inf_{s\le \cdot}x(s)$ is
  continuous on $C([0,t])$. Therefore, by the continuous mapping theorem
\begin{equation}
  \label{rn}
  r^{X^{(n)}}(\cdot)=
\,\frac{R_{[n\cdot]}}{\sqrt{n}}\overset{J_1}{\Rightarrow} r^X(\cdot).
\end{equation}
The tightness of $C^{(n)}$, $n\ge 1$, (\ref{rn}), Lemma~\ref{Cn}, and
the ``convergence together'' result (\cite{Bil}, Theorem 3.1) imply
that $C^{(n)}\overset{J_1}{\Rightarrow} \delta r^X$ as $n\to\infty$.

Now we have a vector-valued sequence of processes $(X^{(n)}, B^{(n)},
C^{(n)})$, $n\ge 1$, that is tight. Therefore, along a subsequence,
this 3-dimensional process converges to $(X,B,\delta r^X)$. Since
$X^{(n)}=B^{(n)}+C^{(n)}$, we get that $X=B+\delta r^X$.
\end{proof}
We shall conclude this section with proofs of
Lemmas~~\ref{Bn}--\ref{LmERWTight}.
\begin{proof}[Proof of Lemma~\ref{Bn}]
  We shall use the functional limit theorem for martingale differences
  (\cite{Bil}, Theorem 18.2). Let
  $\xi_{nk}=n^{-1/2}(\Delta_{k-1}-E_{0,\omega}(\Delta_{k-1}|{\cal
    F}_{k-1}))$, $k,n\in\IN$. Due to rescaling and the
  fact that ERW moves in unit steps, it is obvious that the Lindeberg
  condition, \[\sum_{k\le
    nt}E_{0,\omega}[\xi_{nk}^2\I_{\{|\xi_{nk}|\ge \epsilon\}}]\to
  0\quad\text{ as $ n\to\infty$ for every $t\ge 0$ and $\epsilon>0$},\]
  is satisfied. Thus, we just have to show the convergence of the
  quadratic variation process, i.e.\ for
  $\IP$-a.e.\ $\omega$ for each $t\ge 0$
  \begin{equation}\label{quadvar}
    \sum_{k\le nt}E_{0,\omega}(\xi_{nk}^2|{\cal
      F}_{k-1})=\frac{[nt]}{n}-\frac{1}{n}\sum_{k\le
      nt}\left(E_{0,\omega}(\Delta_{k-1}|{\cal
        F}_{k-1})\right)^2\Rightarrow t
  \end{equation}
  as $n\to\infty$. Since \[0\le \frac{1}{n}\sum_{k\le
    nt}\left(E_{0,\omega}(\Delta_{k-1}|{\cal F}_{k-1})\right)^2\le
  \frac{M}{n}\,R_{[nt]},\] it is enough to prove that
  $P_{0,\omega}(R_{[nt]}>\epsilon n)\to 0$ a.s.\ for each
  $\epsilon>0$. We have \[P_{0,\omega}(R_{[nt]}>\epsilon n)\le
  P_{0,\omega}(T_{[\epsilon n/3]}\le nt)+P_{0,\omega}(T_{-[\epsilon
    n/3]}\le nt)=:f_{n,\epsilon}(\omega,t). \] By Fubini's theorem and
  Lemma~\ref{LmERWRange},
\begin{equation*}
  \IE \left(\sum_{n=1}^\infty f_{n,\epsilon}(\omega,t)\right) =
  \sum_{n=1}^\infty
  \IE f_{n,\epsilon}(\omega,t)=
  \sum_{n=1}^\infty \left(P_0\left(T_{[\epsilon n/3]}\le
      nt\right)+P_0\left(T_{-[\epsilon n/3]}\le
      nt\right)\right) <\infty.
\end{equation*}
This implies that $f_{n,\epsilon}(\omega,t)\to 0$ a.s.\ as $n\to\infty$
and completes the proof.
\end{proof}

\begin{proof}[Proof of Lemma~\ref{Cn}]
  Let $d_m=\sum_{i=1}^M(2\omega_m(i)-1)$ be the total drift stored at
  site $m$, $m\in\IZ$. Then \[C_k-\delta R_k=
  \sum_{m=I_k}^{S_k}(d_m-\delta)-\sum_{m=I_k}^{S_k}\I_{\{L_m(k)<M\}}
  \sum_{j=L_m(k)+1}^M(2\omega_m(j)-1). \] By Lemma~\ref{LmERWRange},
  given $\nu>0$, we can choose $K$ sufficiently large so that
  $P_0(R_{[nt]}>K\sqrt{n})<\nu/2$ for all $n\in\IN$.  We have
  \begin{multline}\label{Cn1}
    P_0\left(\sup_{k\leq nt} \frac{|C_k-\delta R_k|}{\sqrt
        n}>\epsilon\right)\le P_0\left(\max_{k\le
        nt}\frac{\Big|\sum_{m=I_k}^{S_k}(d_m-\delta)\Big|}{R_k}\,
      \frac{R_k}{\sqrt{n}}>\frac{\epsilon}{2},\frac{R_{[nt]}}{\sqrt{n}}\le
      K\right)\\+P_0\left(\frac{M}{\sqrt{n}}\sum_{m=I_{[nt]}}^{S_{[nt]}}
      \I_{\{L_m([nt])<M\}} >\frac{\epsilon}{2}, \frac{R_{[nt]}}{\sqrt{n}}\le
      K\right)+\frac{\nu}{2}.
  \end{multline}
  By the strong law of large numbers $\lim\limits_{(a+b)\to\infty}
  (a+b)^{-1}\sum_{m=-a}^{b}(d_m-\delta)=0$ ($\IP$-a.s.).  Therefore,
  for $\IP$-a.e.\ $\omega$ there is an $r(\omega)\in\IN$ such that
  $R_k^{-1}\Big|\sum_{m=I_k}^{S_k}(d_m-\delta)\Big|\le \epsilon/(2K)$
  whenever $R_k\ge r(\omega)$, and the first term in the right-hand
  side of (\ref{Cn1}) does not exceed
  \begin{equation*}
    P_0\left(\frac{2(M+1)r(\omega)}{\sqrt{n}}> 
      \frac{\epsilon}{2},\frac{R_{[nt]}}{\sqrt{n}}\le
      K\right)\le
    \IE\left(P_{0,\omega}\left(r(\omega)>\frac{\epsilon\sqrt{n}} 
        {4(M+1)}\right)\right)\to 0\ \ \text{as }n\to\infty.
  \end{equation*}
Thus, we only need to estimate the second term in the right-hand side
of (\ref{Cn1}).
  
Divide the interval $[I_{[nt]}, S_{[nt]}]$ into subintervals of
  length $n^{1/4}.$ By Lemma \ref{LmERWEatAll}, given
  $\gamma_1\in(\delta,1)$, with probability at least
  $1-\theta^{n^{\gamma_2/4}} K n^{1/4} $ all subintervals except the
  two extreme ones have at most $n^{\gamma_1/4}$ points which are
  visited less than $M$ times. Hence, for $n$ sufficiently large
  \begin{multline*}
    P_0\left(\frac{M}{\sqrt{n}}\sum_{m=I_{[nt]}}^{S_{[nt]}}
      \I_{\{L_m([nt])<M\}} >\frac{\epsilon}{2}, \frac{R_{[nt]}}{\sqrt{n}}\le
      K\right)\le \\P_0\left(\sum_{m=I_{[nt]}}^{S_{[nt]}}
      \I_{\{L_m([nt])<M\}} >n^{(1+\gamma_1)/4}+2n^{1/4}, \frac{R_{[nt]}}{\sqrt{n}}\le
      K\right)\le \theta^{n^{\gamma_2/4}} K n^{1/4},
  \end{multline*}
and the proof is complete.
\end{proof}

\begin{proof}[Proof of Lemma \ref{LmERWTight}]
  The idea of the proof is the following. If $X^{(n)}$ has large
  fluctuations then either $B^{(n)}$ has large fluctuations or
  $C^{(n)}$ has large fluctuations. $B^{(n)}$ is unlikely to have
  large fluctuations, since it converges to the Brownian motion. 
  By Lemma \ref{LmERWEatAll}, $C_n$ can have large fluctuations only if
  $S_n$ increases or $I_n$ decreases. However by Lemma
  \ref{LmERWRange} neither $I_n$ nor $S_n$ can change too quickly. Let us
  give the details.

  To prove both statements of Lemma~\ref{LmERWTight} it is enough to
  show that there exists $C_3, \alpha>0$ such that for all
  $\ell\in\IN$ and sufficiently large $n,\ n>2^\ell$,
  \begin{equation}
  \label{Holder}
    P_0(\cup_{k<2^\ell} \Omega_{n,k,\ell})\leq C_3 2^{-\alpha \ell},
  \end{equation}
where
\[\Omega_{n,k,\ell}= \left\{\left|X^{(n)}\left(\frac{k+1}{2^\ell}\right) 
    -X^{(n)}\left(\frac{k}{2^\ell}\right)\right|> 2^{-\ell/8}\right\}\]
(see e.g. the last paragraph in the proof of Lemma 1 in \cite{GS1}, Chapter III, Section~5).

Let
\begin{equation}
  \label{notat}
  m_1:=\left[\frac{k n}{2^\ell}\right],\ 
  m_2:=\left[\frac{(k+1) n}{2^\ell}\right],\
  J:=\frac{1}{4}\,n^{1/2}2^{-\ell/8}. 
\end{equation}
Then
\[\Omega_{n,k,\ell}=
\{|X_{m_2}-X_{m_1}|>4J\}\subset \Omega_{n,k,\ell}^{B}\cup\Omega_{n,k,
  \ell}^C,
\]
where 
\[\Omega_{n,k,\ell}^{B}=\{|B_\tau-B_{m_1}|>J,\tau\le m_2\},\quad
\Omega_{n,k, \ell}^C=\{|C_\tau-C_{m_1}|>3J,\tau\le m_2\}, \]  
$\tau:=\inf\{m>m_1\,:\,|X_m-X_{m_1}|>4J\}$ and $B_n$ and $C_n$ are defined in
\eqref{bcd}.

Since
$(B_{j+m_1}-B_{m_1})$, $j\ge 0$, is a martingale, whose quadratic
variation grows at most linearly, the maximal inequality and
Burkholder-Davis-Gundy inequality (\cite{HH}, Theorem 2.11 with $p=4$)
imply that
\[P_{0,\omega}(\Omega_{n,k,\ell}^{B})\leq
  P_{0,\omega}\left(\max_{m_1\le j\le m_2} |B_j-B_{m_1}|>J\right)\le
\frac{C(m_2-m_1)^2} {J^4}\le C'2^{-3\ell/2} .\] Hence,
$P_0\left(\cup_{k<2^\ell} \Omega_{n,k,\ell}^{B}\right)\leq
C'2^{-\ell/2}$ . 

To control $P_0(\Omega_{n,k, \ell}^C)$ consider the following
intervals:
\[ A_1=(-\infty, I_{m_1})\cap\Gamma, \quad A_2=[I_{m_1}, S_{m_1}]\cap
\Gamma,\quad A_3=(S_{m_1}, \infty) \cap \Gamma, \] where
$\Gamma=[X_{m_1}-4J, X_{m_1}+4J]$. Then
\begin{multline*}
  \Omega_{n,k, \ell}^C\subset
  \bigcup_{s=1}^3\left\{\sum_{j=m_1}^{\tau-1}|E_{0,\omega}(\Delta_j\,|\,{\cal
      F}_j)|\I_{\{X_j\in A_s\}}>J,\tau\le m_2\right\} \\ \subset
  \bigcup\limits_{s=1}^3\left\{\sum_{j=m_1}^{m_2-1}|E_{0,\omega}(\Delta_j\,|\,{\cal
      F}_j)|\I_{\{X_j\in
      A_s\}}>J\right\}=:\bigcup\limits_{s=1}^3\Omega_{n,k, \ell,s}^C.
\end{multline*}
To estimate 
$P_0(\Omega^C_{n, k, \ell, 3})$ note that to accumulate a drift larger
than $J$ the walk should visit at least $[J/M]$ distinct sites, i.e.
\[ \Omega^C_{n, k, \ell, 3}\subset \{T_{S_{m_1}+[J/M]}-T_{S_{m_1}+1}\leq
m_2-m_1\}.\] Let $\brJ=[J/(2M)]$ and 
$\brl=\ell/8$.  There exists an $m\in\IN$ such that $S_{m_1}+1\le m \brJ
\le (m+1)\brJ\le S_{m_1}+[J/M].$ Using Lemma~\ref{LmERWRange}, we can
find $K>1$ such that $P_0(S_n>K\sqrt{n})<2^{-\brl}$ for all
sufficiently large $n$. Therefore, \[P_0(\Omega^C_{n, k, \ell, 3})\le
2^{-\brl}+ P_0\left(\cup_{m<2^{\brl+3}M K } \Omega^\dagger_{n,m,
    \ell}, S_n\le K\sqrt{n}\right),\]
where
$\Omega^\dagger_{n,m,\ell}=\left\{T_{(m+1)\brJ}-T_{m\brJ}\leq m_2-m_1
\right\}$. Since $m_2-m_1\le C\brJ^2/2^{6\brl}$ for some constant
$C>0$, Lemma~\ref{LmERWRange} implies that there is $\htheta<1$ such
that and all sufficiently large $n$
\[ P_0\left(\cup_{m<2^{\brl+3}K M}\Omega^\dagger_{n,m,\ell}\right)
\leq \sum_{m<2^{\brl+3}KM}P_0\left(\Omega^\dagger_{n,m,\ell}\right)
\leq 2^{\brl+3}KM \htheta^{2^{3\ell}}<C'' 2^{-\ell}. \] 
$P_0(\cup_{k<2^\ell} \Omega_{n,k,\ell, 1}^C)$ is estimated in 
the same way.

We consider now $A_2$, which is a random subinterval of
$[-m_1,m_1]$ and, on  $\Omega_{n,k, \ell,2}^C$, has
length between $J/M$ and $8J$. To estimate $P_0(\Omega_{n,k,
  \ell,2}^C)$ we notice that by Lemma~\ref{LmERWEatAll}, outside of an
event of exponentially small (in $J^{\gamma_2}$) probability,  the number of
cookies that are left in $A_2$ at time $m_1$ does not exceed
$CJ^{\gamma_1}$, where $\gamma_1<1$. Even if the walker consumes all
cookies in that interval, it can not build up a
drift of size $J\gg CJ^{\gamma_1}$ (for $J$ large). With this idea in mind,
we turn now to a formal proof.

As we noted above, on $\Omega_{n,k, \ell,2}^C$, we have  $A_2\in {\cal I}$,
where ${\cal I}$ denotes the set of all intervals of the form
\[[a,b],\  a,b\in\IZ, \ -m_1\le a< b\le m_1,\ J/M\le b-a\le
8J.\] The cardinality of ${\cal I}$ does not exceed $16m_1J\le
Cn^{3/2}$. Therefore,
\begin{equation}
  \label{max}
  P_0(\Omega_{n,k, \ell,2}^C)\le
Cn^{3/2}\max_{A\in {\cal I}}
P_0\left(\sum_{j=m_1}^{m_2-1}|E_{0,\omega}(\Delta_j\,|\,{\cal
    F}_j)|\I_{\{X_j\in I\}}>J,A_2=A\right). 
\end{equation}
By the definition of $A_2$, the walk necessarily crosses the interval
$A_2$ by the time $m_1$. The leftover drift in $A_2$ is at most $M$
times the number of sites in $A_2$, which still have at least one
cookie. Writing $A$ as $[a,b]$, $a,b\in\IZ$, $a<b$, we can estimate
the last probability by \[P_0\left(M\sum_{m=a}^b\I_{\{L_m(T_a\vee
    T_b)<M\}}>J\right)=P_0\left(\sum_{m=a}^b\I_{\{L_m(T_a\vee
    T_b)<M\}}>J/M\right).\] If $a\ge 0$ we can apply
Lemma~\ref{LmERWEatAll} and get that for all sufficiently large $n$
(such that $(8J)^{\gamma_1}\le J/M$)
\begin{multline}
  \label{easy}
P_0\left(\sum_{m=a}^b\I_{\{L_m(T_a\vee T_b)<M\}}>J/M\right)\le
P_0\left(\sum_{m=a}^b\I_{\{L_m(T_b)<M\}}>(b-a)^{\gamma_1}\right)\\\le
\theta^{(b-a)^{\gamma_2}}\le \theta^{(J/M)^{\gamma_2}}.
\end{multline}
The case $b\le 0$ is similar.
Finally, consider the case $a<0<b$. Then
\begin{multline}
  \label{also}  
  P_0\left(\sum_{m=a}^b\I_{\{L_m(T_a\vee T_b)<M\}}>J/M\right)\le
  P_0\left(\sum_{m=a}^0\I_{\{L_m(T_a)<M\}}>J/(2M)\right)\\+P_0\left(\sum_{m=0}^b\I_{\{L_m(
      T_b)<M\}}>J/(2M)\right).
\end{multline}
If $b\le J/(2M)$ then the last term in (\ref{also}) is $0$. But for
$J/(2M)< b\le 8J$ we have that $b^{\gamma_1}\le J/(2M)$ for all
sufficiently large $J$. Lemma~\ref{LmERWEatAll} implies that
\[P_0\left(\sum_{m=0}^b\I_{\{L_m(T_b)<M\}}>J/(2M)\right)\le
P_0\left(\sum_{m=0}^b\I_{\{L_m(T_b)<M\}}>b^{\gamma_1}\right)\le
\theta^{b^{\gamma_2}}\le \theta^{(J/(2M))^{\gamma_2}}.\] The first
term in the right-hand side of (\ref{also}) is estimated in the same
way. We conclude that for some constant $C$ and all sufficiently large
$n$
\[P_0(\cup_{k<2^\ell}\Omega_{n,k, \ell,2}^C)\le Cn^{3/2}2^\ell
\theta^{(J/(2M))^{\gamma_2}}<2^{-\ell}.\]

This completes the proof of \eqref{Holder}
establishing Lemma~\ref{LmERWTight}.
\end{proof}

\section{Boundary case: Proof of Theorem~\ref{ThERWCrit}.}

Let $\delta=1$. For $t\ge 0$ and $n\ge 2$
set \[T^{(n)}(x):=\frac{T_{[nx]}}{n^2/\log^2 n},\quad
X^{(n)}(t):=\frac{X_{[nt]}}{\sqrt{n}\log n},\quad
S^{(n)}(t):=\frac{S_{[nt]}}{\sqrt{n}\log n}.\] Let $\Sigma_j,\ j\ge 0$
be i.i.d.\ positive integer-valued random variables defined in
(\ref{sigmas}). They satisfy (\ref{progeny}) with $\delta=1$ and by
\cite[Chapter 9, Section 6]{GS} for some constant
$a>0$
\begin{equation}
  \label{cflt}
  \frac{\sum_{j=0}^{[n\cdot]}\Sigma_j}{n^2}\overset{J_1}{\Rightarrow}
aH(\cdot)\quad\text{as }n\to\infty,
\end{equation}
where $H:=(H(x)),\ x\ge 0$, is a stable subordinator with
index $1/2$. More precisely,
\begin{equation}
  \label{HTP}
  H(x)=\inf\{t\ge 0:\, B(t)=x\}.
\end{equation}
We shall need the following two lemmas.
\begin{lemma}
  \label{fdd} The finite dimensional distributions of $T^{(n)}$
  converge to those of $cH$, where $c>0$ is a constant and
  $H$ is given by (\ref{HTP}).
\end{lemma}
\begin{lemma}
  \label{btrack}
  For every $\epsilon>0$, $T>0$ \[\lim_{n\to\infty}P_0\left(\sup_{0\le
      t\le T}(S^{(n)}(t)-X^{(n)}(t))>\epsilon\right)=0.\]
\end{lemma}
Theorem~\ref{ThERWCrit} is an easy consequence of these lemmas.
\begin{proof}[Proof of Theorem~\ref{ThERWCrit}]
  Lemma~\ref{fdd} implies that the finite dimensional distributions of
  the process $S^{(n)}$ converge to those of $DB^*$, where $D>0$ is a
  constant. Since the trajectories of $S^{(n)}$ are monotone and the
  limiting process $B^*$ is continuous, we conclude that $S^{(n)}$
  converges weakly to $DB^*$ in the (locally) uniform topology (see
  \cite{A}, Corollary~1.3 and Remark (e) on p.\,588). Finally, by
  Lemma~\ref{btrack} for each $T>0$ \[\sup_{0\le t\le
    T}(S^{(n)}(t)-X^{(n)}(t))\to 0\] in $P_0$ probability. By the
  ``converging together'' theorem (\cite[Theorem 3.1]{Bil}) we
  conclude that $X^{(n)}$ converges weakly to $DB^*$ in the (locally)
  uniform topology, and, thus, in $J_1$.
\end{proof}
\begin{proof}[Proof of Lemma~\ref{fdd}]
  Let $k\in\IN$ and $0=x_0<x_1<\dots<x_k$. We have to show that for any
  $0=t_0<t_1<t_2<\dots<t_k$
  \begin{multline*}
    P_0(T^{(n)}(x_k)-T^{(n)}(x_i)\le
  t_{k-i},\ \forall i=0,1,2,\dots,k-1)\\\to P(T(x_k)-T(x_i)\le t_{k-i},\
  \forall i=0,1,\dots,k-1),\ \text{as }n\to\infty,
  \end{multline*}
where $T(\cdot)=cH(\cdot)$ for some $c>0$.

  At time $T_{[nx_k]}$ consider the structure of the corresponding
  branching process as we look back from $[nx_k]$. Notice that
  $D_{[nx_i],j}\le D_{[nx_k],j}$ for $i\le k$ and all $j$. This simple
  observation will allow us to get bounds on $T_{[nx_i]}$,
  $i=1,2,\dots,k-1$, in terms of the structure of downcrossings at time
  $T_{[nx_k]}$. This means that we can use the same copy of the
  branching process $V$ to draw conclusions about all
  hitting times $T_{[nx_i]}$, $i=1,2,\dots,k$.

  We shall use notation (\ref{sigmas}) and
  let $N^{(0)}=0$, \[N^{(k-i)}=\min\{m\in
  \IN:\,\sigma_m\ge [nx_k]-[nx_i]\},\quad
  i=0,1,2,\dots,k-1.\] 
Since \[2\sum_{j=1}^{N^{(k-i)}-1}\Sigma_j\le T_{[nx_k]}-T_{[nx_i]}\le
nx_k-nx_i+2\sum _{j=1}^{N^{(k-i)}} \Sigma_j,\] we have
  \begin{multline}\label{upper}
    P_0(T^{(n)}(x_k)-T^{(n)}(x_i)\le
  t_{k-i},\ \forall i=0,1,2,\dots,k-1)\\\le P\left(2\sum_{j=1}^{N^{(k-i)}-1}
  \Sigma_j\le n^2t_{k-i}/\log^2 n,\ \forall i=0,1,2,\dots,k-1\right)
  \end{multline}
and 
\begin{multline}\label{lower}
  P_0(T^{(n)}(x_k)-T^{(n)}(x_i)\le t_{k-i},\ \forall
  i=0,1,2,\dots,k-1)\\\ge P\left([nx_k]-[nx_i]+2\sum_{j=1}^{N^{(k-i)}}
  \Sigma_j\le n^2t_{k-i}/\log^2 n,\ \forall i=0,1,2,\dots,k-1\right).
\end{multline}
Next we provide some control on $N^{(k-i)}$, $i=0,1,\dots,k-1$, and on
the maximal lifetime over $[nx_k]$ generations. Theorem~\ref{LmERWLongExc}
and \cite[Theorem 3.7.2]{Du} imply
that $\sigma_n/(n\log n)\Rightarrow b^{-1}$ for some positive constant
$b$. From this it is easily seen
that
\begin{equation}
  \label{num}
  \frac{\min\{m\in\IN:\,\sigma_m>n\}}{nb/\log
    n}\Rightarrow 1 \quad\text{as }n\to\infty.
\end{equation}
Recalling our definition of $N^{(k-i)}$ 
we get that for every $\epsilon,\nu>0$ there is
$n_0$ such that for all $n\ge n_0$ \[P\left(1-\nu\le
  \frac{N^{(k-i)}}{\Bar{N}^{(k-i)}} \le 1+\nu,\
  i=0,\dots,k-1\right)>1-\epsilon,\] where $\Bar{N}^{(k-i)}=b(x_k-x_i)n/\log
n$. In particular, for $C=(1+\nu)bx_k$ we have that \[P\left(N^{(k)}\le
  \frac{Cn}{\log n}\right)>1-\epsilon.\] Define $\lambda_n=(\log
n)^{-1/2}$ (any sequence $\lambda_n$, $n\in\IN$, such that
$\lambda_n\to 0$ and $ \lambda_n\log n \to\infty$ will work) and
notice that by Theorem~\ref{LmERWLongExc} there is $n_1$ such that for
all $n\ge n_1$ \[P\left(\max_{1\le i\le Cn/\log n}(\sigma_i-\sigma_{i-1})\le
n\lambda_n\right)\ge \left(1-\frac{2C_1}{n\lambda_n}\right)^{Cn/\log
  n}>1-\epsilon.\] Thus, on a set $\Omega_\epsilon$ of measure at
least $1-2\epsilon$ for all $n\ge n_0\vee n_1$ the number of lifetimes of the
branching process $V$ covering $[nx_k]-[nx_i]$ generations,
$i=0,1,2,\dots,k-1$, is well controlled and the maximal lifetime over
$[nx_k]$ generations does not exceed $n\lambda_n$. In particular, on
$\Omega_\epsilon$, the number of lifetimes in any interval
$([nx_i],[nx_{i+1}]),\ i=0,1,\dots,k-1$, goes to infinity as $n\to\infty$. 

Finally, on $\Omega_\epsilon$ we get from (\ref{upper}) and
(\ref{cflt}) that
\begin{multline*}
  P_0(T^{(n)}(x_k)-T^{(n)}(x_i)\le t_{k-i},\ \forall
  i=0,1,2,\dots,k-1)\\\le
  P\left(2\sum_{j=1}^{(1-\nu)\Bar{N}^{(k-i)}-1} \Sigma_j\le
    n^2t_{k-i}/\log^2 n,\ \forall i=0,1,2,\dots,k-1\right)\\=
  P\left(\frac{\sum_{j=1}^{(1-\nu)\Bar{N}^{(k-i)}-1}
      \Sigma_j}{((1-\nu)n/\log n)^2}\le \frac{t_{k-i}}{2(1-\nu)^2},\
    \forall i=0,1,2,\dots,k-1\right)\\\to P(aH(b(x_k-x_i))\le
  (1-\nu)^{-2}t_{k-i}/2\ \forall
  i=0,1,2,\dots,k-1)\\=P(2ab^2(H(x_k)-H(x_i))\le
  t_{k-i}(1-\nu)^{-2})\ \forall i=0,1,2,\dots,k-1).
\end{multline*}
The lower bound is shown starting from (\ref{lower}) in exactly the
same way.  Letting $\nu\to 0$ and then $\epsilon\to 0$ we obtain the
statement of the lemma with $T(\cdot)=2ab^2H(\cdot)=:cH(\cdot)$.
\end{proof}

\begin{proof}[Proof of Lemma~\ref{btrack}]
  Without loss of generality we can consider $t\in[0,1]$. Fix some
  $\nu>0$.  We have
\begin{multline}\label{btrack1}
  P_0\left(\sup_{0\le t\le 1}(S^{(n)}(t)-X^{(n)}(t))>\epsilon\right)
  \le P_0(S_n\ge K\sqrt{n}\ln n)+\\P_0\left(\max_{0\le m\le
      n}(S_m-X_m)>\epsilon \sqrt{n}\ln n, S_n< K\sqrt{n}\ln
    n)\right).
\end{multline}
By Lemma~\ref{fdd} we can find $K>0$ such that for all large $n$
\[P_0(S_n\ge K\sqrt{n}\ln n)\le P_0(T_{[K\sqrt{n}\ln n]}\le n)<\nu.\] To
estimate the last term in (\ref{btrack1}) we shall use properties of the branching
process $V$. Let $N=\min\{m\in\IN:\,\sigma_m>
K\sqrt{n}\ln n\}$. Then the last term in (\ref{btrack1}) is bounded by
\begin{multline*}
  P_0^V\left(\max_{i\le N} (\sigma_i-\sigma_{i-1})\ge
    \epsilon\sqrt{n}\ln n\right)\le \\ P_0^V(N>C\sqrt{n})+P_0^V\left(\max_{i\le
      C\sqrt{n}} (\sigma_i-\sigma_{i-1})\ge \epsilon\sqrt{n}\ln n,
    N\le C\sqrt{n}\right)\overset{(\ref{num})}{\le}\\ \nu
  +P_0^V\left(\max_{i\le C\sqrt{n}} (\sigma_i-\sigma_{i-1})\ge \epsilon\sqrt{n}\ln n\right),
\end{multline*}
for some large $C$ and all sufficiently large $n$. Finally, from
(\ref{lifetime}) we conclude that for all large enough $n$ the last
probability does not exceed \[1-\left(1-\frac{2C_1}{\epsilon\sqrt{n}\ln
    n}\right)^{[C\sqrt{n}]}<\nu.\] This completes the proof.
\end{proof}

\noindent \textbf{Acknowledgments.}  The authors are grateful to the
Fields Institute for Research in Mathematical Sciences for support and
hospitality. D.~Dolgopyat was partially supported by the NSF grant DMS
0854982. E.~Kosygina was partially supported by a Collaboration Grant
for Mathematicians (Simons Foundation) and the PSC CUNY Award \#
64603-00 42. The authors also thank the anonymous referee for careful
reading of the paper and remarks which helped to improve the exposition.

\bigskip

{\sc \small
\begin{tabular}{ll}
Department of Mathematics& \hspace*{30mm}Department of Mathematics\\
University of Maryland& \hspace*{30mm}Baruch College, Box B6-230\\
4417 Mathematics Building&\hspace*{30mm}One Bernard Baruch Way\\
College Park, MD 20742, USA &\hspace*{30mm}New York, NY 10010, USA\\
{\verb+dmitry@math.umd.edu+}& \hspace*{30mm}{\verb+elena.kosygina@baruch.cuny.edu+}
\end{tabular}\vspace*{2mm}

}

\end{document}